\documentclass{article}%
\usepackage{amsmath}
\usepackage{amsfonts}
\usepackage{amssymb}
\usepackage{graphicx}
\usepackage{amscd}%
\providecommand{\U}[1]{\protect\rule{.1in}{.1in}}
\newtheorem{theorem}{Theorem}

\newtheorem{lemma}{Lemma}

\newenvironment{proof}[1][Proof]{\noindent\textbf{#1.} }{\ \rule{0.5em}{0.5em}}
\begin{document}

\title{Holomorphic plane fields with many invariant hypersurfaces.}
\author{L. C\^{a}mara and B. Sc\'{a}rdua}
\maketitle

\begin{abstract}
In this paper we show that a (non necessarily integrable) holomorphic plane
field on a compact complex manfold $M$ having an infinite number of invariant
hypersurfaces must admit a meromorphic first integral $F:M\longrightarrow
\mathbb{C}^{p}$. In particular, it is integrable.

\end{abstract}

\section{Holomorphic plane fields}

\paragraph{Invariant hyperplanes for plane fields.}

In \cite{Ghys} E. Ghys generalized a result by J.-P. Jouanolou
(\cite{Jouanolou}) showing that any holomorphic Pfaff equation on a compact
complex manifold has only a finite number of compact leaves unless all leaves
are compact. Both results are generalized versions of a classical result due
to G. Darboux (cf. \cite{Darboux},\cite{Jouanolou2}) for holomorphic
foliations in dimension two. Our main goal in this paper is to generalize this
statement to plane fields of greater codimensions. For this sake, we shall
need to revise and introduce some concepts.

Recall from de Medeiros work that a holomorphic plane field of codimension $p$
is locally given by an LDS $p$-form (\cite{deMed}, Proposition 1.2.1, p. 455),
i.e., a germ of holomorphic $p$-form $\omega$ locally decomposable off the
singular set $\operatorname*{Sing}(\omega)$ with $\operatorname*{codim}%
(\operatorname*{Sing}(\omega))\geq2$. Thus a holomorphic plane field of
codimension $p$ on a complex manifolds $M$ is given by the following data:

\begin{enumerate}
\item An open covering $\mathcal{U}=\{U_{\alpha}\}$ of $M$;

\item A collection of holomorphic LDS $p$-forms $\omega_{\alpha}\in
\Omega_{\mathcal{O}}^{p}(U_{\alpha})$ whose singular set $\operatorname*{Sing}%
(\omega)$ has codimension $\geq2$;

\item A collection of maps $g_{\beta\alpha}\in\mathcal{O}^{\ast}%
(U_{\alpha\beta})$ such that $\omega_{\beta}=g_{\beta\alpha}\omega_{\alpha}$
whenever $U_{\alpha\beta}=U_{\alpha}\cap U_{\beta}\neq\emptyset$.
\end{enumerate}

The Cech cocycle $g=(g_{\beta\alpha})\in H^{1}(\mathcal{U},\mathcal{O}^{\ast
})$ defines a line bundle $\mathcal{L}$ over $M$. Thus the collection
$(\omega_{\alpha})$ may be interpreted as a global holomorphic $p$-form
$\omega=(\omega_{\alpha})\in H^{0}(\mathcal{U},\Omega_{\mathcal{O}}^{p}%
\otimes\mathcal{L})$ defined over $M$ and assuming values on the line bundle
$\mathcal{L}$. The singular set $\operatorname*{Sing}(\mathcal{F}_{\omega})$
is supposed to be the union of the singular sets of $\omega_{\alpha}$ and has
codimension greater or equal to $2$.

Recall that the algebraic characterization of the LDS$\ p$-forms is given by
the equation $i_{v_{I}}\omega\wedge\omega=0$ for any $p$-vector field
$v_{I}=v_{1}\wedge\cdots\wedge v_{n}$, where $\{v_{1},\cdots,v_{p}\}$ is a
local frame. Further, the integrability of such plane fields may be
characterized by the additional condition $i_{v_{I}}\omega\wedge d\omega=0$
\ (\cite{deMed}, Proposition 1.2.2, p. 455).

A codimension $p$ plane field $\omega=0$ is said to admit a \emph{first
integral} if it is tangent to the the level sets of a holomorphic map
$f:M\longrightarrow\mathbb{C}^{p}$. In such case clearly this plane field is
integrable (in the previous sense) and its leaves are contained in the level
sets of $f$. Since in general a compact complex manifold has few holomorphic
functions, then we often discuss the meromorphic integrability of such
foliations. Clearly such object must be a map $f=(f_{1},\cdots,f_{p}%
):M\longrightarrow\mathbb{C}^{p}$, such that each $f_{j}$ is a meromorphic
function, whose level sets are tangent to the plane field $\omega=0$. In this
case, all the solutions are obviously complete intersections.

Recall that in the codimension $1$ case an infinite number of compact
solutions leads to the existence of a (global) meromorphic first integral.
Thus it is natural to investigate under what conditions on its invariant
hypersurfaces a holomorphic plane field of codimension $p$ admits a
meromorphic first integral.

Recall that a hypersurface (a codimension $1$ variety) is given by an open
covering $\mathcal{U}=\{U_{\alpha}\}$ of $M$ and complex analytic functions
$f_{\alpha}:U_{\alpha}\longrightarrow\mathbb{C}$ such that $f_{\beta}%
=g_{\beta\alpha}f_{\alpha}$ for some $(g_{\beta\alpha})\in H^{1}%
(\mathcal{U},\mathcal{O}^{\ast})$. This last condition means essentially that
$f_{\alpha}$ and $f_{\beta}$ vanish at the same place in $U_{\alpha}\cap
U_{\beta}$. A\ hypersurface $V=(f=0)$ is said to be \emph{invariant} by
$\omega$ if $\operatorname*{Ker}(\omega(p))\subset\operatorname*{Ker}df(p)$
for all $p\in V$. This coincides with the usual geometric definition in case
$\omega$ is integrable, i.e., the tangent space of the solutions are tangent
to $V$.

\begin{theorem}
\label{hol. plane field 1}Let $M$ be a compact complex manifold and $\omega\in
H^{0}(M,\Omega^{p}\otimes\mathcal{L})$ be induced by LDS $p$-forms. Then the
$p$-codimension plane field $(\mathcal{F}:\omega=0)$ admits a meromorphic
first integral iff $\mathcal{F}$ admits  infinitely many invariant hypersurfaces.
\end{theorem}

\paragraph{Chern classes and divisors.}

Recall that a divisor is a (formal) linear combination of varieties
$a_{1}f^{1}+\cdots+a_{k}f^{k}$ representing the variety $(f^{1})^{a_{1}}%
\cdots(f^{k})^{a_{k}}=0$, $a_{j}\in\mathbb{Z}$. This defines a homomorphism
$\mathfrak{Div}(M)\longrightarrow$ $\check{H}^{1}(M,\mathcal{O}^{\ast})$ from
the set of analytic divisors in $M$ to the set of line bundles on $M$. From
the exact sequence $0\longrightarrow\mathbb{Z}\hookrightarrow\mathcal{O}%
\overset{\exp}{\longrightarrow}\mathcal{O}^{\ast}\longrightarrow0$ one can
define the first Chern class $c_{1}(v)\in\check{H}^{2}(M,\mathbb{C}%
)\simeq\check{H}^{1}(\mathcal{U},\Omega_{\mathcal{O}}^{1})$ of a divisor
$v\in\mathfrak{Div}(M)$ as the first Chern class of $(g_{\beta\alpha}%
)\in\check{H}^{1}(\mathcal{U},\mathcal{O}^{\ast})$. Considering the de Rham
isomorphism $\check{H}^{2}(M,\mathbb{C})\simeq\check{H}^{1}(\mathcal{U}%
,\Omega_{\mathcal{O}}^{1})\simeq H^{2}(M,\mathbb{C})$ one can write the
representative of $c_{1}(v)$ in $\check{H}^{1}(\mathcal{U},\Omega
_{\mathcal{O}}^{1})$ as $d\log g_{\beta\alpha}$. Thus we obtain the
homomorphism  $\mathfrak{Div}(M)\longrightarrow\check{H}^{1}(\mathcal{U}%
,\mathcal{Z}_{\mathcal{O}}^{1})$ given by $v\mapsto c_{1}(v)=d\log
g_{\beta\alpha}$, where $(g_{\beta\alpha})$ is the line bundle determined by
the divisor $v$ and $\mathcal{Z}_{\mathcal{O}}^{1}$ denotes the sheaf of germs
of closed holomorphic $1$-forms. Now consider the $\mathbb{C}$-linear map
$\psi:\mathfrak{Div}(M)\otimes\mathbb{C}\longrightarrow\check{H}%
^{1}(\mathcal{U},\mathcal{Z}_{\mathcal{O}}^{1})$ and denote its Kernel by
$\mathfrak{Div}_{0}(M)$. Clearly each $v\in\mathfrak{Div}_{0}(M)$ is of the
form $v=\sum\lambda_{j}v_{j}$ where each $v_{j}:(f_{j}=0)$  is the closure of
an invariant hypersurface of $\mathcal{F}$. Note that the linear structure of
the map $\psi$ and the existence of infinitely many invariant hypersurfaces
ensure the existence of infinitely many invariant divisors in the kernel of
$\psi$, i.e., infinitely many \emph{flat} invariant divisors.

\paragraph{First integrals and $p$--divisors.}

Suppose $c_{1}(v)$ vanishes, then refining the covering if necessary one may
find $(\zeta_{\alpha})\in\check{H}^{0}(\mathcal{U},\mathcal{Z}_{\mathcal{O}%
}^{1})$ such that $c_{1}(v)=\zeta_{\beta}-\zeta_{\alpha}$. Since with
$f_{\alpha}^{j}=g_{\alpha\beta}^{j}\cdot f_{\beta}^{j}$, then $\sum_{j=1}%
^{n}\lambda_{j}d\log f_{\alpha}^{j}-\lambda_{j}d\log f_{\beta}^{j}%
=\zeta_{\beta}-\zeta_{\alpha}$. We have just proved that

\begin{lemma}
[\cite{Ghys}]\label{meromorphic form}For each $v=\sum\lambda_{j}v_{j}%
\in\mathfrak{Div}_{0}$, where $v_{j}$ is an invariant hypersurface of
$\mathcal{F}$ given by $v:(f_{j}=0)$, there is a global closed meromorphic
$1$-form $\xi=(\xi_{\alpha})\in\check{H}^{0}(\mathcal{U},\mathcal{Z}%
_{\mathcal{M}}^{1})$ given by $\xi_{\alpha}=\sum_{j=1}^{n}\lambda_{j}d\log
f_{\alpha}^{j}+\zeta_{\alpha}$.
\end{lemma}

Notice that the global closed meromorphic $1$-form $\xi=(\xi_{\alpha}%
)\in\check{H}^{0}(\mathcal{U},\mathcal{Z}_{\mathcal{M}}^{1})$ is unique up to
adding a global closed holomorphic $1$-form.

In the sequel we shall need the following technical result well known for $1$-forms.

\begin{lemma}
\label{invar.}Let $\omega$ be a holomorphic $p$-form on $U\subset
\mathbb{C}^{n}$ admitting the hypersurface $V:(f=0)$ as invariant set. Then
there exists a holomorphic $(p+1)$-form $\varpi$ such that%
\[
\omega\wedge d\log(f)=\varpi.
\]

\end{lemma}

\begin{proof}
Since $\ker(\omega(x))\subset\ker(df(x))$ for all $x\in V$, then $df(x)\in
\ker(\omega(x))^{\bot}$ (cf. \cite{God1969}). From \cite{deMed}, Proposition
1.2.2, one has $df(x)\in\mathcal{E}^{\ast}(\omega(x))$ for all $x\in V$. In
other words, $\omega(x)\wedge df(x)=0$ $x\in V$. The result then follows from
R\"{u}ckert nullstellensatz (the analytic version of Hilbert's nullstellensatz).
\end{proof}

From Lemmas \ref{meromorphic form} and \ref{invar.} we obtain the linear map%
\[%
\begin{array}
[c]{cccc}%
\phi: & \mathfrak{Div}_{0}(M) & \longrightarrow & \check{H}^{0}(\mathcal{U}%
,\Omega_{\mathcal{O}}^{p}\otimes\mathcal{L})/\omega\wedge\check{H}%
^{0}(\mathcal{U},\mathcal{Z}_{\mathcal{O}}^{2})\\
& \left(  v=\sum_{j=1}^{n}\lambda_{j}v_{j}\right)   & \mapsto & [\omega
\wedge\xi]
\end{array}
\]
where $\mathfrak{Div}_{0}$ denotes the set of \thinspace flat divisors (i.e.,
those with vanishing first Chern class).

Since $\phi$ is linear, $\dim(\check{H}^{0}(\mathcal{U},\Omega_{\mathcal{O}%
}^{p}\otimes\mathcal{L}))$ is finite, and, by hypothesis, $\mathfrak{Div}%
_{0}(M)$ has infinite dimension, we may obtain $p$ linearly independent
divisors $(v_{i})$, $i=1,\cdots,p$, in the kernel of $\phi$. Thus we have a
closed meromorphic $1$-form $\xi$ such that $\xi\wedge\omega=0$. Repeating
this process, we may find linearly independent closed meromorphic $1$-forms
$\xi_{i_{1}},\cdots,\xi_{i_{p}}$ as above. Thus $\xi_{I}=\xi_{i_{1}}%
\wedge\cdots\wedge\xi_{i_{p}}$ is a closed meromorphic $p$-form. Let $\eta
_{I}=\xi_{i_{1}}\wedge\cdots\wedge\xi_{i_{p}}$, since it is decomposable, then
$\mathcal{E}^{\ast}(\eta_{I})=\mathcal{I}(\xi_{i_{1}},\cdots,\xi_{i_{p}})$,
thus $\mathcal{E}^{\ast}(\eta_{I}(x))\subset\mathcal{E}^{\ast}(\omega(x))$ for
all $x$ off the polar set of $\eta_{I}$, say $P(\eta_{I})$. From de Medeiros
division Lemma (\cite{deMed}, Lemma 3.1.1, p. 460) there are functions $h_{I}$
holomorphic off $P(\eta_{I})$ such that $\omega=$ $h_{I}\eta_{I}$. As a
consequence, there are global meromorphic functions $h_{IJ}$ such that
$\xi_{I}=h_{IJ}\xi_{J}$. Since $\xi_{I}$ is closed, then $0=dh_{IJ}\wedge
\xi_{J}$. In particular, $0=dh_{IJ}\wedge\omega$. Choosing the $I$ and $J$
properly (since $\mathfrak{Div}_{0}(M)$ has infinite dimension), we may find
the desired global meromorphic first integral $h:M\longrightarrow
\mathbb{C}^{p}$.

\section{Transversely holomorphic structures}

The notion of a smooth manifold endowed with a transversely holomorphic
structure was introduced in the work of X. Gomez-Mont (\cite{Gomez-Mont}).
Behind all the construction is the notion of structural sheaf. More precisely,
we suppose the existence of a smooth manifold $M$ admitting a real smooth
foliation $\mathcal{F}$ of real codimension $2p$ with holomorphic transversal
sections and transversely holomorphic transition maps. The \emph{structural
sheaf} $\mathcal{O}$ being the germ of functions constant along the leaves and
transversely holomorphic. The structural sheaf is also called the sheaf of
\emph{basic transversely holomorphic functions}. In a similar fashion one may
define the sheaf of \emph{basic transversely meromorphic functions
}$\mathcal{M}$, \emph{basic transversely holomorphic} $p$\emph{-forms}
$\Omega_{\mathcal{O}}^{p}$, the \emph{basic transversely meromorphic }%
$p$\emph{-forms} $\Omega_{\mathcal{M}}^{p}$ etc.

Therefore, one may construct structures similar to the ones present in the
classical holomorphic case like divisors and associate line bundles. Using as
key point the finiteness of $\dim_{\mathbb{C}}(\check{H}^{0}(\mathcal{U}%
,\Omega_{\mathcal{O}}^{p+1}\otimes\mathcal{L}))$ (cf. \cite{Gomez-Mont}),
Brunella and Nicolau (\cite{Brunella-Nicolau}) repeated essentially the same
arguing in (\cite{Ghys}) in order to prove the existence of non constant basic
transversely meromorphic functions for any transversely holomorphic structure
of complex codimension $1$ defined on a compact and connected manifold $M$
admitting infinitely many compact solutions.

Since Gomez-Mont theory is constructed for transversely holomorphic structures
of any complex codimension, it is natural to study the same problem for
greater codimensions. Again we need to introduce some notation. The first
remark is that we are dealing with foliations without singularities, thus a
transversely holomorphic foliation of complex codimension $p$ on a complex
manifolds $M$ is given by the following data:

\begin{enumerate}
\item An open covering $\mathcal{U}=\{U_{\alpha}\}$ of $M$;

\item A collection of non-singular, basic transversely holomorphic, locally
decomposable, and integrable $p$-forms $\omega_{\alpha}\in\Omega_{\mathcal{O}%
}^{p}(U_{\alpha})$;

\item A collection of maps $g_{\beta\alpha}\in\mathcal{O}^{\ast}%
(U_{\alpha\beta})$ such that $\omega_{\beta}=g_{\beta\alpha}\omega_{\alpha}$
whenever $U_{\alpha\beta}=U_{\alpha}\cap U_{\beta}\neq\emptyset$.
\end{enumerate}

Again the Cech cocycle $g=(g_{\beta\alpha})\in H^{1}(\mathcal{U}%
,\mathcal{O}^{\ast})$ defines a line bundle $\mathcal{L}$ over $M$. Thus the
collection $(\omega_{\alpha})$ may be interpreted as a global basic
transversely holomorphic $p$-form $\omega=(\omega_{\alpha})\in H^{0}%
(\mathcal{U},\Omega_{\mathcal{O}}^{p}\otimes\mathcal{L})$ defined over $M$ and
assuming values on the line bundle $\mathcal{L}$.

A transversely holomorphic hypersurface (a codimension $1$ variety) is given
by an open (trivializing) covering $\mathcal{U}=\{U_{\alpha}\}$ of $M$ and
complex functions $f_{\alpha}:U_{\alpha}\longrightarrow\mathbb{C}$ such that
$f_{\beta}=g_{\beta\alpha}f_{\alpha}$ for some $(g_{\beta\alpha})\in
H^{1}(\mathcal{U},\mathcal{O}^{\ast})$. This last condition means essentially
that $f_{\alpha}$ and $f_{\beta}$ vanish at the same place in $U_{\alpha}\cap
U_{\beta}$ with transversely holomorphic structure. Notice that $f=(f_{\alpha
})$ is not necessarily basic. A\ hypersurface $V:(f=0)$ is said to be
\emph{invariant} by $\omega$ if $\operatorname*{Ker}(\omega(p))\subset
\operatorname*{Ker}df(p)$ for all $p\in V$. This coincides with the usual
geometric definition in case $\omega$ is integrable, i.e., the tangent space
of the solutions are tangent to $V$.

Again a linear combination with complex coefficients $v=a_{1}v_{1}%
+\cdots+a_{n}v_{n}$ is a \emph{divisor} if each $a_{j}\in\mathbb{Z}$ and each
$v_{j}$ is a transversely holomorphic hypersurface representing the variety
$(f^{1})^{a_{1}}\cdots(f^{k})^{a_{k}}=0$. In a similar fashion, we obtain the
linear map $\psi:$ $\mathfrak{Div}(M)\otimes\mathbb{C}\longrightarrow\check
{H}^{1}(\mathcal{U},\Omega_{\mathcal{O}}^{1})$ given by $v=\sum_{j=1}%
^{n}\lambda_{j}v_{j}\mapsto\sum_{j=1}^{n}\lambda_{j}c_{1}(v_{j})=\sum
_{j=1}^{n}\lambda_{j}d\log g_{\beta\alpha}^{j}$, where $(g_{\beta\alpha}^{j})$
is the line bundle determined by the invariant hypersurface $v_{j}$.

Once again the linear structure of the map $\psi$ and the existence of
infinitely many invariant transversely holomorphic hypersurfaces ensure the
kernel of $\psi$ has sinfinite dimension, i.e., the $\mathbb{C}$-linear space
$\mathfrak{Div}_{0}(M)$ of \emph{flat} invariant divisors has infinite
dmension. From adapted versions of Lemmas \ref{meromorphic form} and
\ref{invar.} we obtain the following result.

\begin{theorem}
\label{trans. hol. fol.}Let $M$ be a smooth manifold endowed with a transverse
holomorphic structure of complex codimension $p$ given by $\omega\in
H^{0}(M,\Omega^{p}\otimes\mathcal{L})$. Then the $p$-codimension plane field
$(\mathcal{F}:\omega=0)$ admits a basic transversely meromorphic first
integral $F:M\longrightarrow\mathbb{C}^{p}$ iff $\mathcal{F}$ admits
infinitely many invariant transversely holomorphic hypersurfaces.
\end{theorem}

The proof is essentially the same as the one for the holomorphic case. In fact
it is even simpler due to the absence of singularities. For the smooth version
of de Medeiros division Lemma is almost immediate.

\end{document}